\documentclass[letterpaper,11pt,twoside,keywordsasfootnote,addressatend,noinfoline]{article}
\usepackage{fullpage}
\usepackage[english]{babel}
\usepackage{amssymb}
\usepackage{amsmath}
\usepackage{theorem}
\usepackage{epsfig}
\usepackage{subfigure}
\usepackage{multirow}
\usepackage{imsart}
\usepackage{color}

\theorembodyfont{\normalfont}
\newtheorem{theorem}{Theorem}

\newtheorem{lemma}[theorem]{Lemma}

\newenvironment{proof}{\noindent{\scshape Proof.}}{\hspace*{2mm} $\square$}

\newcommand{\Z}{\mathbb{Z}}
\newcommand{\R}{\mathbb{R}}
\newcommand{\ind}{\mathbf{1}}
\newcommand{\ep}{\epsilon}

\newcommand{\ax}{\eta}
\newcommand{\rw}{\zeta}
\newcommand{\pile}{\xi}
\newcommand{\near}{A}
\newcommand{\next}{B}
\newcommand{\scalar}[1]{\langle #1 \rangle}

\DeclareMathOperator{\card}{card}
\DeclareMathOperator{\uniform}{Uniform \,}

\DeclareMathOperator{\geometric}{Geometric \,}
\DeclareMathOperator{\weight}{weight \,}

\begin{document}

\begin{frontmatter}

\title     {Fixation results for the two-feature Axelrod model with a variable number of opinions}
\runtitle  {Fixation for the Axelrod model with a variable number of opinions}
\author    {Nicolas Lanchier\thanks{Research supported in part by NSF Grant DMS-10-05282.} and Paul-Henri Moisson}
\runauthor {Nicolas Lanchier and Paul-Henri Moisson}
\address   {School of Mathematical and Statistical Sciences, \\ Arizona State University, \\ Tempe, AZ 85287, USA. \\ E-mail: nlanchie@asu.edu}
\address   {Centre de Math\'ematiques appliqu\'ees, \\ \'Ecole Polythechnique, \\ F-91128 Palaiseau Cedex, France. \\ E-mail: paul-henri.moisson@polytechnique.edu}

\begin{abstract} \ \
 The Axelrod model is a spatial stochastic model for the dynamics of cultures that includes two key social mechanisms:
 homophily and social influence, respectively defined as the tendency of individuals to interact more frequently with
 individuals who are more similar and the tendency of individuals to become more similar when they interact.
 The original model assumes that individuals are located on the vertex set of an interaction network and are characterized by
 their culture, a vector of opinions about~$F$ cultural features, each of which offering the same number~$q$ of alternatives.
 Pairs of neighbors interact at a rate proportional to the number of cultural features for which they agree, which results
 in one more agreement between the two neighbors.
 In this article, we study a more general and more realistic version of the standard Axelrod model that allows for a variable
 number of opinions across cultural features, say~$q_i$ possible alternatives for the $i$th cultural feature.
 Our main result shows that the one-dimensional system with two cultural features fixates when~$q_1 + q_2 \geq 6$.
\end{abstract}

\begin{keyword}[class=AMS]
\kwd[Primary ]{60K35}
\end{keyword}

\begin{keyword}
\kwd{Interacting particle systems, Axelrod model, random walks, fixation.}
\end{keyword}

\end{frontmatter}


\section{Introduction}
\label{sec:intro}

\indent This paper is concerned with the Axelrod model~\cite{axelrod_1997} for the dissemination of cultures, probably the most
 popular stochastic model of culture dynamics.
 The model includes explicit space in the form of local interactions using the framework of interacting particle systems:
 individuals are located on the set of vertices of a graph whose edges represent potential dyadic interactions.
 This work focuses on the one-dimensional lattice where each individual can only interact with her nearest left and right neighbors.
 Individuals are characterized by their culture, a vector of opinions about various cultural features, rather than a single opinion
 like in the voter model~\cite{clifford_sudbury_1973, holley_liggett_1975}.
 Specifically, the model is a continuous-time Markov chain whose state at time~$t$ is a function
\begin{equation}
\label{eq:state-space-1}
  \ax_t : \Z \longrightarrow \{1, 2, \ldots, q \}^F = \,\hbox{set of cultures}
\end{equation}
 with the integers~$F$~and~$q$ denoting respectively the number of cultural features and the common number of possible
 opinions per cultural feature.
 The dynamics is dictated by what has been identified as the two most important social mechanisms:
\begin{itemize}
 \item {\bf homophily} which is defined as the tendency of individuals to interact more frequently with individuals who are more similar and \vspace*{4pt}
 \item {\bf social influence} which is defined as the tendency of individuals to become more similar as the result of their interactions.
\end{itemize}
 Note that the set of cultures is equipped with a natural distance:
 the function that counts the number of disagreements between two cultures.
 This distance is the key to modeling both homophily and social influence:
 homophily by assuming that neighbors interact at a rate that decreases with the distance between their cultures and social influence by
 assuming that the result of an interaction is to decrease the cultural distance between the neighbors.
 Specifically, in the Axelrod model, pairs of nearest neighbors interact at a rate equal to the fraction of cultural features for which they
 agree and, as a result of an interaction, one of the two neighbors chosen at random mimics the other neighbor for one of the cultural features
 for which they disagree (if any).
 In particular, given that two neighbors disagree about exactly~$j$ cultural features, one given neighbor mimics the other one for one given
 cultural feature for which they disagree at rate
 $$ r (j) := (1/2)(1/j)(1 - j/F) \ \ \hbox{when} \ \ j \neq 0 \qquad \hbox{and} \qquad r (0) \ := \ 0. $$
 The left-hand side is the probability that one given neighbor rather than the other one updates her culture times the probability that any of
 the~$j$ cultural features for which they disagree is the one chosen for update times the rate at which both individuals indeed interact, while
 the right-hand side is simply a convention based on the fact that, when both neighbors already agree on all cultural features, the interaction
 has no effect.
 Therefore, letting
 $$ \begin{array}{rrl}
    \scalar{\ax_t (x)}_i & := & \hbox{$i$th coordinate of the vector} \ \ax_t (x) \ \hbox{for} \ i = 1, 2, \ldots, F \vspace*{2pt} \\
                         &  = & \hbox{opinion of the individual at vertex~$x$ for the $i$th cultural feature}, \end{array} $$
 and denoting the Hamming distance by
 $$ H (\ax (x), \ax (y)) := \card \,\{i = 1, 2, \ldots, F : \scalar{\ax (x)}_i \neq \scalar{\ax (y)}_i \} $$
 the dynamics is described by the transition rates
\begin{equation}
\label{eq:transition}
  \begin{array}{l}
    \lim_{h \to 0} \,(1/h) \,P \,(\scalar{\ax_{t + h} (x)}_i = j \ | \ \scalar{\ax_t (x)}_i \neq j) \vspace*{4pt} \\ \hspace*{100pt} = \
    \sum_{y = x \pm 1} \ r (H (\ax_t (x), \ax_t (y))) \ \ind \{\ax_t (y) = j \} \end{array}
\end{equation}
 for all~$(x, i) \in \Z \times \{1, 2, \ldots, F \}$ and opinion~$j \in \{1, 2, \ldots, q \}$. \vspace*{5pt}


\noindent {\bf Previous mathematical results} --
 The Axelrod model has been extensively studied both heuristically and through numerical simulations in the past~15~years while its
 analytical study has been initiated more recently in~\cite{lanchier_2012}.
 The main question that has been explored mathematically is the dichotomy between fluctuation and fixation, namely whether individuals
 change their culture infinitely often or a finite number of times.
 More precisely, we say that the system
 $$ \begin{array}{rcl}
    \hbox{fluctuates when} & \card \,\{t : \scalar{\ax_{t-} (x)}_i \neq \scalar{\ax_t (x)}_i \} = \infty & \hbox{a.s. for all $x$ and $i$}  \vspace*{4pt} \\
       \hbox{fixates when} & \card \,\{t : \scalar{\ax_{t-} (x)}_i \neq \scalar{\ax_t (x)}_i \} < \infty & \hbox{a.s. for all $x$ and $i$}. \end{array} $$
 Note that whether fluctuation or fixation occurs is very sensitive to the initial distribution.
 Also, to fix the ideas, we assume from now on that the system starts from the most natural distribution: the product measure in
 which all the cultures are equally likely, i.e.,
 $$ P \,(\scalar{\ax_0 (x)}_i = j) = q^{-1} \quad \hbox{for all} \quad j = 1, 2, \ldots, q. $$
 The first mathematical result about the Axelrod model, established in~\cite{lanchier_2012}, states that the two-feature two-opinion
 system fluctuates and clusters:
 $$ \begin{array}{l} \lim_{t \to \infty} \,P \,(\ax_t (x) = \ax_t (y)) = 1 \quad \hbox{for all} \quad x, y \in \Z. \end{array} $$
 The proof is based on a coupling between the cultural dynamics and systems of annihilating random walks, such coupling being obtained
 by putting particles between neighbors to keep track of their disagreements.
 Using this coupling together with delicate symmetry arguments modeled after the construction given in~\cite{adelman_1976}, fluctuation
 and clustering have been extended in~\cite{lanchier_schweinsberg_2012} to the process with any finite number of cultural features but
 again only two opinions per cultural feature.
 Increasing the number of opinions, the particles that keep track of the disagreements can not only annihilate but also coalesce.
 In addition, regardless of the number of opinions, there are active particles that jump at a positive rate and mark the boundary between
 neighbors that can interact, and frozen particles that cannot jump and mark the boundaries between neighbors that cannot interact.
 The main result in~\cite{lanchier_scarlatos_2013} states that the system fixates whenever
\begin{equation}
\label{eq:fixation-condition}
  F/q \ < \ (1 - 1/q)^{F - 1}
\end{equation}
 which is proved by counting and comparing the initial numbers of active and frozen particles in a large interval.
 Some basic algebra shows that~\eqref{eq:fixation-condition} holds whenever
 $$  F \leq cq \ \ \hbox{where} \ \ c \approx 0.567 \ \ \hbox{satisfies} \ \ e^{-c} = c. $$
 Refining some of the arguments to obtain~\eqref{eq:fixation-condition}, it is also proved that the system with two cultural features
 and three opinions per feature fixates which, together with~\cite{lanchier_2012}, gives a complete picture of the model with two features.
 See Figure~\ref{fig:diagram} for a summary of these results.
 The process has also been studied on the two-dimensional torus in~\cite{li_2014} when both the number of cultural features and number
 of opinions per feature are large.
 There, it is proved that, for all $\ep > 0$, the system reaches a consensus with probability close to one on a giant connected component whenever
 $$ F/q \ \geq \ \ln (2) + \ep \quad \hbox{and} \quad q \ > \ q_0 (\ep). $$
 That is, there is a connected component that covers a positive fraction of the graph and in which all the vertices share the same culture
 eventually. \vspace*{5pt}

\begin{figure}[t]
\centering
\scalebox{0.36}{\input{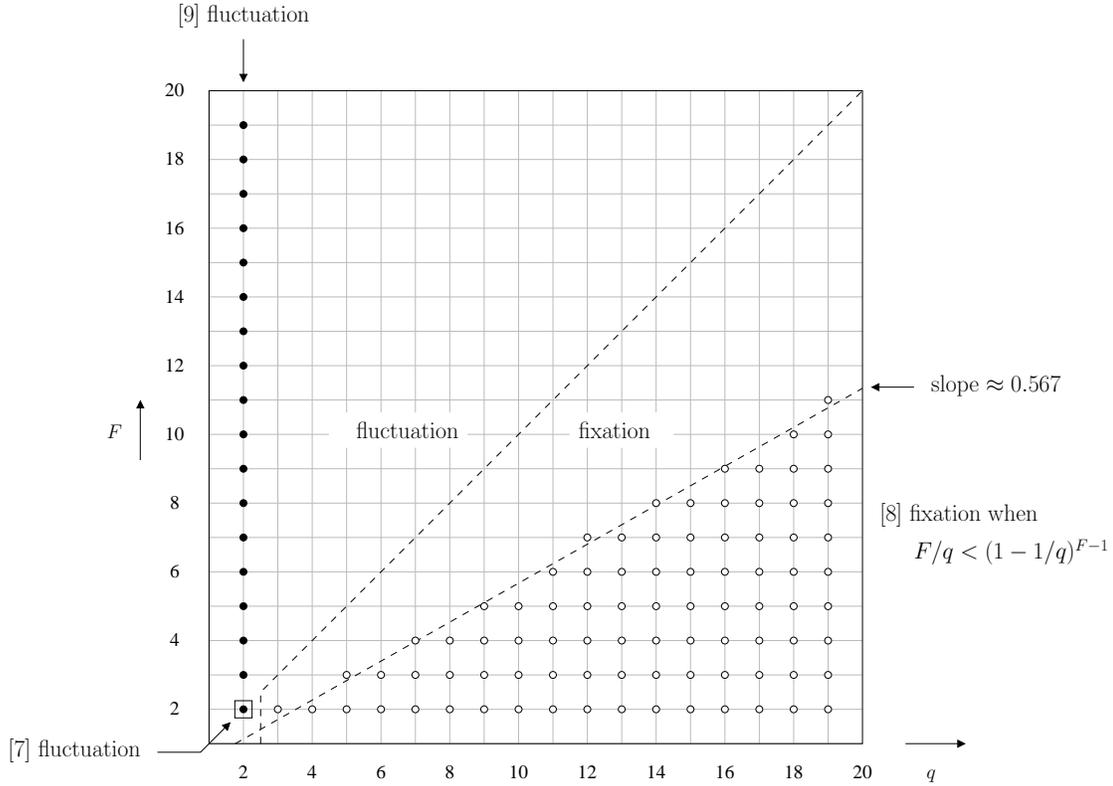}}
\caption{\upshape{Phase diagram of the one-dimensional Axelrod model in the $q - F$ plane.
 The black dots represent the set of parameters for which fluctuation has been proved in~\cite{lanchier_2012, lanchier_schweinsberg_2012} whereas
 the white dots represent the set of parameters for which fixation has been proved in~\cite{lanchier_scarlatos_2013}.}}
\label{fig:diagram}
\end{figure}


\noindent {\bf Variable number of opinions} --
 This work is motivated by the simple observation that the assumption on the fixed number of opinions across cultural features seems
 unrealistic.
 Just to give a concrete example, if there are two cultural features called politics and religion, there is no reason for the number
 of candidates at the next election to be equal to the number of possible religious beliefs.
 To define a more general model with a variable number of opinions, we simply assume that the state space of the process is given
 by the set of functions
\begin{equation}
\label{eq:state-space-2}
  \begin{array}{l} \ax_t : \Z \longrightarrow \{1, 2, \ldots, q_1 \} \times \{1, 2, \ldots, q_2 \} \times \cdots \times \{1, 2, \ldots, q_F \}. \end{array}
\end{equation}
 Since the transition rates~\eqref{eq:transition} do not depend on~$q$, they again describe the dynamics of the model with a variable
 number of opinions.
 In particular, we now consider the model with state space and local transition rates~\eqref{eq:state-space-2}~and~\eqref{eq:transition},
 and starting from the product measure with
\begin{equation}
\label{eq:initial}
  P \,(\scalar{\ax_0 (x)}_i = j) = q_i^{-1} \quad \hbox{for all} \quad j = 1, 2, \ldots, q_i.
\end{equation}
 In this more general setting, the particles that keep track of the disagreements between neighbors again evolve like
 annihilating-coalescing random walks where active particles jump at rates that depend on the number of disagreements.
 The techniques we develop to study the generalized model requires active particles to all jump at the same rate, which is the case
 only when there are two cultural features.
 In this case, the results from~\cite{lanchier_2012, lanchier_scarlatos_2013} give a complete picture of the long-term behavior when
 the number of opinions per feature is constant: the system
 $$ \begin{array}{rl}
    \hbox{fluctuates when} & F \ = \ 2 \quad \hbox{and} \quad q_1 \ = \ q_2 \ = \ 2  \vspace*{2pt} \\
       \hbox{fixates when} & F \ = \ 2 \quad \hbox{and} \quad q_1 \ = \ q_2 \ > \ 2. \end{array} $$
 Our main result extends the fixation region.
\begin{theorem} --
\label{th:fixation}
 Assume that~$F = 2$~and~$q_1 + q_2 \geq 6$. Then, the system fixates.
\end{theorem}
 Note that our theorem together with~\cite{lanchier_2012} gives a complete picture of the general system with two features
 and a variable number of opinions except when~$q_1 + q_2 = 5$. \vspace*{5pt}


\noindent {\bf Structure of the proof} --
 The rest of the paper is devoted to the proof of the theorem.
 First, we explain how to construct the process from a so-called graphical representation and give a rigorous definition of the
 system of particles that keeps track of the disagreements between neighbors.
 This system of particles being coupled with the cultural model, it can be constructed from the same graphical representation,
 which is also used to prove that it evolves according to a certain system of annihilating-coalescing random walks.
 The dynamics of this system of random walks has already been described in~\cite{lanchier_scarlatos_2013} but only heuristically.
 In contrast, we give a rigorous proof for each of the transition rates.
 In addition to the coupling between the Axelrod model and annihilating-coalescing random walks, there are two key ingredients
 to prove our fixation result:
\begin{itemize}
 \item The first ingredient is a construction due to Bramson and Griffeath~\cite{bramson_griffeath_1989} based on duality-like
   techniques to obtain an implicit condition for fixation in terms of the initial number of active and frozen particles
   in a large interval.
   Their argument has been extended in~\cite{lanchier_scarlatos_2013} to systems where the state at each vertex is a vector,
   which is the result we use. \vspace*{4pt}
 \item The second ingredient is a monotonicity relationship between the random number of collisions among the particles that keep
   track of the disagreements for some cultural feature and the number of possible opinions for this feature.
   This relationship is irrelevant to understand the standard model with a fixed number of opinions across cultural
   features but it is crucial in our general context where the number of opinions is variable.
\end{itemize}
 Having these two ingredients in hands, the rest of the proof is mostly technical.
 First, combining these ingredients with the law of large numbers, we deduce a weak version of our theorem:
 the one-dimensional process fixates when~$q_1 + q_2 \geq 7$.
 The estimates obtained to prove this weak version only account for the initial distribution assuming a worst case scenario for
 the realization of the system of particles.
 These estimates are then improved by also accounting for specific collision events which, together with the ergodic theorem,
 gives the full fixation result.


\section{Graphical representation}
\label{sec:representation}

\indent The Axelrod model falls into the general class of interacting systems considered in~\cite{harris_1972} therefore the process
 is well-defined and can be constructed starting from any initial configuration using a so-called graphical representation.
 In the case of the Axelrod model, this graphical representation consists of a random graph involving independent Poisson processes marking the
 times of potential interactions and additional collections of independent Bernoulli random variables and uniform random variables to determine
 the outcome of each interaction.
 More precisely, \vspace*{4pt} \\
 for all pairs vertex-cultural feature~$(x, i) \in \Z \times \{1, 2, \ldots, F \}$,
\begin{itemize}
 \item we let $(N_{x, i} (t) : t \geq 0)$ be independent rate one Poisson processes, \vspace*{4pt}
 \item we denote by $T_{x, i} (n)$ the $n$th arrival time: $T_{x, i} (n) := \inf \,\{t : N_{x, i} (t) = n \}$, \vspace*{4pt}
 \item we let $(B_{x, i} (n) : n \geq 1)$ be collections of independent Bernoulli variables with
  $$ P \,(B_{x, i} (n) = - 1) \ = \ P \,(B_{x, i} (n) = + 1) \ = \ 1/2, $$
 \item and we let $(U_{x, i} (n) : n \geq 1)$ be collections of independent $\uniform (0, 1)$.
\end{itemize}
 Then, at each time $t := T_{x, i} (n)$, we draw an arrow
 $$ (y, i) := (x + B_{x, i} (n), i) \ \to \ (x, i) $$
 and call this arrow {\bf active} if and only if
\begin{equation}
\label{eq:active}
  \scalar{\ax_{t-} (x)}_i \neq \scalar{\ax_{t-} (y)}_i  \quad  \hbox{and} \quad U_{x, i} (n) \ \leq \ 2 \times r (H (\ax_{t-} (x), \ax_{t-} (y))).
\end{equation}
 In words, arrows in the graphical representation mark the times of potential interactions whereas active arrows correspond to the random subset
 of these arrows that indeed result in an interaction and an update of the system.
 In particular, the one-dimensional Axelrod model can be constructed from the graphical representation above by setting
\begin{equation}
\label{eq:update}
 \scalar{\ax_t (x)}_i \ := \ \scalar{\ax_{t-} (y)}_i \quad \hbox{whenever} \quad t := T_{x, i} (n) \ \hbox{for some $n$ and \eqref{eq:active} is satisfied}
\end{equation}
 which indeed produces the desired local transition rates in~\eqref{eq:transition}.


\section{Coupling with annihilating-coalescing random walks}
\label{sec:coupling}

\indent The first key ingredient, introduced in~\cite{lanchier_2012} and improved
 in~\cite{lanchier_scarlatos_2013, lanchier_schweinsberg_2012}, to study the one-dimensional Axelrod model is a coupling between the
 cultural dynamics and a certain system of annihilating-coalescing random walks that keeps track of the disagreements between nearest
 neighbors.
 In this section, we define this coupling and give a rigorous proof of the evolution rules of the system of random walks that has been
 described heuristically in~\cite{lanchier_scarlatos_2013}.
 To begin with, we visualize the culture of each individual as a column of~$F$ dots where the dot at level~$i$ can have~$q_i$
 different colors corresponding to the~$q_i$ possible states for the~$i$th cultural feature.
 Since the particles that keep track of the disagreements between nearest neighbors evolve on the set of edges rather than the set of
 vertices, it is convenient to identify edges with their midpoint and to define translations on the set of edges and vertices as follows:
 $$ \begin{array}{rclcl}
              e \ := \ (x, x + 1) & \equiv & x + 1/2     & \hbox{for} & x \in \Z \vspace*{4pt} \\
      e + a \ := \ (x, x + 1) + a & \equiv & x + 1/2 + a & \hbox{for} & (e, a) \in (\Z + 1/2) \times (\Z/2). \end{array} $$
 To keep track of the disagreements between neighbors, we then set
\begin{equation}
\label{eq:spin}
  \begin{array}{rcl}
  \rw_t (e, i) \ := \ \ind \,\{\scalar{\ax_t (e - 1/2)}_i \neq \scalar{\ax_t (e + 1/2)}_i \} & \hbox{for each pair edge-level} & (e, i). \end{array}
\end{equation}
 This defines a spin system that we visualize by putting a particle at each pair edge-level which is in state~1 and we refer
 to Figure~\ref{fig:coupling} for a picture of this coupling for the stochastic process with four cultural features.
 The number of particles per edge, defined as
 $$ \begin{array}{rcl} \pile_t (e) \ := \ \sum_{i = 1, 2, \ldots, F} \ \rw_t (e, i) & \hbox{for each edge} & e \in \Z + 1/2 \end{array} $$
 is a key quantity to understand the dynamics of this spin system since it counts the number of disagreements between neighbors which, in
 turn, is related to the rate at which these individuals interact.
 The next four lemmas give together a full description of the dynamics induced by our coupling on the spin system~\eqref{eq:spin}.
 To state these lemmas, we introduce the notations
 $$ \begin{array}{rl}
          (x, i) \to_t (x \pm 1, i) & \hbox{to indicate that} \vspace*{0pt} \\ & \hbox{there is an arrow from $(x, i)$ to $(x \pm 1, i)$ at time~$t$} \vspace*{4pt} \\
     (x, i) \leadsto_t (x \pm 1, i) & \hbox{to indicate that} \vspace*{0pt} \\ & \hbox{there is an active arrow from $(x, i)$ to $(x \pm 1, i)$ at time~$t$}. \end{array} $$
 The first lemma gives an expression of the probability that an arrow in the graphical representation is active using directly the spin system
 rather than the Axelrod model.
\begin{figure}[t]
\centering
\scalebox{0.45}{\input{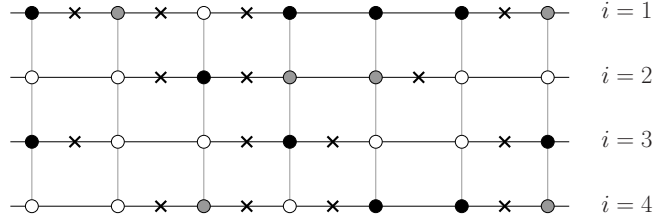}}
\caption{\upshape{Coupling between the Axelrod model and the spin system~\eqref{eq:spin}.
  The state of each individual is represented by white, grey and black dots, while particles in the spin system are represented by crosses.}}
\label{fig:coupling}
\end{figure}
\begin{lemma} --
\label{lem:active}
 For each pair edge-level~$(e, i)$,
 $$ \begin{array}{l}
      P \,((e - 1/2, i) \leadsto_t (e + 1/2, i) \ | \vspace*{4pt} \\ \hspace*{50pt}
           (e - 1/2, i) \to_t (e + 1/2, i) \ \hbox{and} \ \rw_{t-} (e, i) = 1) \ = \ 2 \times r (\pile_{t-} (e)). \end{array} $$
\end{lemma}
\begin{proof}
 First, we observe that
 $$ \begin{array}{rcl}
    \pile_{t-} (e) & = & \card \,\{i : \rw_{t-} (e, i) = 1 \} \ = \ \card \,\{i : \scalar{\ax_{t-} (e - 1/2)}_i \neq \scalar{\ax_{t-} (e + 1/2)}_i \} \vspace*{4pt} \\
                   & = &  H (\ax_{t-} (e - 1/2), \ax_{t-} (e + 1/2)) \end{array} $$
 hence the arrow~$(e - 1/2, i) \to_t (e + 1/2, i)$ is active if and only if
\begin{equation}
\label{eq:active-1}
  \rw_{t-} (e, i) = 1 \quad  \hbox{and} \quad U_{e + 1/2, i} (n) \ \leq \ 2 \times r (\pile_{t-} (e)).
\end{equation}
 In particular, using that the Poisson processes and random variables in the graphical representation are independent and
 that the graphical representation after time~$t-$ is independent of the configuration at time~$t-$, we deduce that
 $$ \begin{array}{l}
      P \,((e - 1/2, i) \leadsto_t (e + 1/2, i) \ | \ (e - 1/2, i) \to_t (e + 1/2, i) \ \hbox{and} \ \rw_{t-} (e, i) = 1) \vspace*{4pt} \\ \hspace*{25pt} = \
      P \,(U_{e + 1/2, i} (n) \leq 2 \times r (\pile_{t-} (e))) \ = \ 2 \times r (\pile_{t-} (e)). \end{array} $$
 This completes the proof.
\end{proof} \\ \\
 In order to describe the dynamics of the spin system, the next step is to understand the effect of an active arrow on the particles.
 First, we note that
 $$ \begin{array}{rcl}
      (x - 1, i) \leadsto_t (x, i) & \hbox{implies that} & \scalar{\ax_t (x')}_{i'} = \scalar{\ax_{t-} (x')}_{i'} \quad \hbox{for all} \quad (x', i') \neq (x, i) \end{array} $$
 which, in terms of the spin system~\eqref{eq:spin}, becomes
 $$ \begin{array}{l}
      (e - 1/2, i) \leadsto_t (e + 1/2, i) \vspace*{4pt} \\ \hspace*{20pt}
    \hbox{implies that} \quad \rw_t (e', i') \ = \ \rw_{t-} (e', i') \quad \hbox{for all} \quad (e', i') \notin \{(e, i), (e + 1, i) \}. \end{array} $$
 In particular, we only need to determine whether the pairs~$(e, i)$~and~$(e + 1, i)$ are empty or occupied just after the interaction.
 Recall from~\eqref{eq:active-1} that, given the active arrow in the statement of the previous lemma, the pair~$(e, i)$ is occupied just before the interaction.
 The next lemma shows that the effect of this active arrow is to make the pair empty with probability one.
\begin{lemma} --
\label{lem:vanish}
 For each pair edge-level~$(e, i)$,
 $$ P \,(\rw_t (e, i) = 0 \ | \ (e - 1/2, i) \leadsto_t (e + 1/2, i)) \ = \ 1. $$
\end{lemma}
\begin{proof}
 In view of the effect~\eqref{eq:update} of an active arrow, we have
 $$ \begin{array}{l}
      P \,(\rw_t (e, i) = 0 \ | \ (e - 1/2, i) \leadsto_t (e + 1/2, i)) \vspace*{4pt} \\ \hspace*{20pt} = \
      P \,(\scalar{\ax_t (e - 1/2)}_i = \scalar{\ax_t (e + 1/2)}_i \ | \ (e - 1/2, i) \leadsto_t (e + 1/2, i)) \ = \ 1. \end{array} $$
 This completes the proof.
\end{proof} \\ \\
 Now, to determine whether the pair~$(e + 1, i)$ is empty or occupied just after the occurrence of the active arrow, we distinguish two cases
 depending on whether this pair is empty or occupied just before the interaction.
 In the next lemma, we show that, in case the pair is empty just before the interaction, it becomes occupied.
 This, together with the previous lemma, indicates that, in this case, there is a jump of a particle in the direction of the active arrow.
\begin{lemma} --
\label{lem:jump}
 For each pair edge-level~$(e, i)$,
 $$ \begin{array}{l}
      P \,(\rw_t (e + 1, i) = 1 \ | \ (e - 1/2, i) \leadsto_t (e + 1/2, i) \ \hbox{and} \ \rw_{t-} (e + 1, i) = 0) \ = \ 1. \end{array} $$
\end{lemma}
\begin{proof}
 Using, as in the previous lemma, the effect~\eqref{eq:update} of an active arrow together with the fact that the simultaneous occurrence of arrows in the
 graphical representation is a negligible event, we deduce that, given the conditioning in the statement of the lemma,
 $$ \scalar{\ax_t (e + 1/2)}_i \ \neq \ \scalar{\ax_{t-} (e + 1/2)}_i \ = \ \scalar{\ax_{t-} (e + 3/2)}_i \ = \ \scalar{\ax_t (e + 3/2)}_i $$
 with probability one, since
 $$ \begin{array}{rcl}
    \rw_{t-} (e + 1, i) \ = \ 0 & \hbox{implies that} & \scalar{\ax_{t-} (e + 1/2)}_i \ = \ \scalar{\ax_{t-} (e + 3/2)}_i. \end{array} $$
 In particular, it follows that
 $$ \begin{array}{l}
      P \,(\rw_t (e + 1, i) = 1 \ | \ (e - 1/2, i) \leadsto_t (e + 1/2, i) \ \hbox{and} \ \rw_{t-} (e + 1, i) = 0) \vspace*{4pt} \\ \hspace*{40pt} = \
      P \,(\scalar{\ax_t (e + 1/2)}_i \neq \scalar{\ax_t (e + 3/2)}_i \ | \vspace*{4pt} \\ \hspace*{95pt}
                         (e - 1/2, i) \leadsto_t (e + 1/2, i) \ \hbox{and} \ \rw_{t-} (e + 1, i) = 0) \ = \ 1 \end{array} $$
 which proves the lemma.
\end{proof} \\ \\
 The last step is to determine whether the pair~$(e + 1, i)$ is empty or occupied just after the interaction given that this pair is occupied by a
 particle just before the interaction, which we interpret respectively as a jump of a particle in the direction of the active arrow and a collision
 with another particle that may cause both particles to either {\bf annihilate} or {\bf coalesce}.
 The answer is simple when we know the background configuration of the Axelrod model and we have
 $$ \begin{array}{rcl}
    \hbox{annihilation} & \hbox{when} & \scalar{\ax_{t-} (e - 1/2)}_i \ = \ \scalar{\ax_{t-} (e + 3/2)}_i \vspace{4pt} \\
    \hbox{coalescence}  & \hbox{when} & \scalar{\ax_{t-} (e - 1/2)}_i \ \neq \ \scalar{\ax_{t-} (e + 3/2)}_i \end{array} $$
 but the problem is made challenging by the fact that the configuration of the spin system only gives us a partial knowledge of the configuration
 of the Axelrod model.
 However, using duality-like techniques and the fact that the initial states are independent, we can prove that successive collisions result
 independently in either annihilation or coalescence with some probabilities that can be computed explicitly, which is done in the next lemma.
 For an illustration of some of the arguments in the proof, we refer the reader to~Figure~\ref{fig:paths}.
\begin{lemma} --
\label{lem:outcome}
 For each pair edge-level~$(e, i)$,
 $$ \begin{array}{l}
      P \,(\rw_t (e + 1, i) = 0 \ | \ (e - 1/2, i) \leadsto_t (e + 1/2, i) \ \hbox{and} \ \rw_{t-} (e + 1, i) = 1) \ = \ (q_i - 1)^{-1}. \end{array} $$
\end{lemma}
\begin{proof}
\begin{figure}[t]
\centering
\scalebox{0.50}{\input{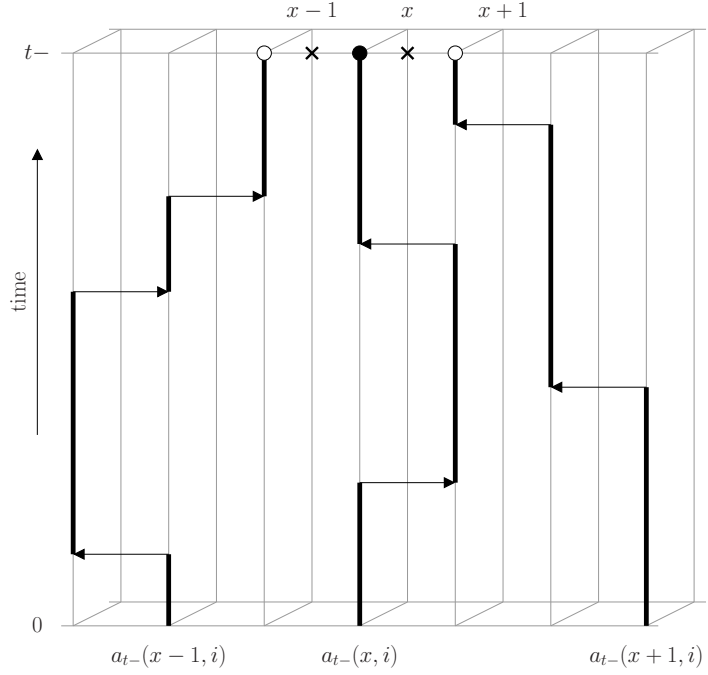}}
\caption{\upshape Picture related to the proof of Lemma~\ref{lem:outcome}.
  The presence of the two particles at the top of the picture implies that the individual at vertex~$x$ disagrees with her two neighbors at time~$t-$ on the $i$th feature,
  which further implies that the corresponding three ancestors originate from three different vertices at time zero.
  In particular, the states of the three individuals at time~$t-$ are independent uniform random variables.}
\label{fig:paths}
\end{figure}
 This is similar to the proof of \cite[Lemma 3]{lanchier_scarlatos_2013}.
 To begin with, we define active paths in order to keep track of the origin of an opinion going backwards in time:
 we say that there exists an {\bf active~$i$-path} from point~$(z, s)$ to point~$(x, t)$ whenever there exist
 $$ s_0 \ = \ s \ < \ s_1 \ < \ \cdots \ < \ s_{n + 1} \ = \ t \qquad \hbox{and} \qquad
    x_0 \ = \ z, \,x_1, \,\ldots, \,x_n \ = \ x $$
 such that the following two conditions hold:
\begin{enumerate}
 \item For all $j = 1, 2, \ldots, n$, there is an active arrow $(x_{j - 1}, i) \leadsto_{s_j} (x_j, i)$. \vspace*{4pt}
 \item For all $j = 0, 1, \ldots, n$, and $s \in (s_j, s_{j + 1})$, there is no active arrow $(x_{j - 1}, i) \leadsto_s (x_j, i)$.
\end{enumerate}
 We write this event~$(z, s) \overset{i}{\leadsto} (x, t)$ and observe that
 $$ \hbox{for all $(x, t) \in \Z \times \R_+$ there exists a unique $z \in \Z$ such that $(z, 0) \overset{i}{\leadsto} (x, t)$}. $$
 In addition, recalling~\eqref{eq:active} and using a simple induction, we have
 $$ \scalar{\ax_t (x)}_i \ = \ \scalar{\ax_0 (z)}_i \quad \hbox{whenever} \quad (z, 0) \overset{i}{\leadsto} (x, t) $$
 so we write~$z = a_t (x, i)$ and call~$z$ the {\bf ancestor} of~$(x, t)$ for the~$i$th cultural feature.
 To prove the lemma using the concept of active path, the first ingredient is to observe that, due to one-dimensional nearest neighbor interactions,
 active paths at the same level~$i$ cannot cross each other, so the dynamics preserve the order of the ancestors at each level:
\begin{equation}
\label{eq:outcome-1}
  a_s (x - 1, i) \ \leq \ a_s (x, i) \ \leq \ a_s (x + 1, i) \quad \hbox{for} \quad s \geq 0 \quad \hbox{and} \quad x := e + 1/2.
\end{equation}
 Moreover, given the conditioning in the statement of the lemma, there is one particle on each side of vertex~$x$ at level~$i$
 just before time~$t$ from which it follows that
\begin{equation}
\label{eq:outcome-2}
  \begin{array}{l}
    X_{\pm 1} \ := \ \scalar{\ax_0 (a_{t-} (x \pm 1, i)}_i \ = \ \scalar{\ax_{t-} (x \pm 1)}_i \vspace*{4pt} \\ \hspace*{80pt}
              \neq \ \scalar{\ax_{t-} (x)}_i \ = \ \scalar{\ax_0 (a_{t-} (x, i))}_i \ =: \ X_0. \end{array}
\end{equation}
 This implies that the inequalities in~\eqref{eq:outcome-1} are strict:
\begin{equation}
\label{eq:outcome-3}
  a_s (x - 1, i) \ < \ a_s (x, i) \ < \ a_s (x + 1, i) \quad \hbox{for} \quad s \geq 0 \quad \hbox{and} \quad x := e + 1/2
\end{equation}
 showing in particular that all three ancestors are different.
 Combining~\eqref{eq:outcome-2}--\eqref{eq:outcome-3} and using that the initial states are independent and uniformly distributed imply that
\begin{equation}
\label{eq:outcome-4}
  X_{-1} \ \hbox{and} \ X_{+1} \ \hbox{are independent Uniform} \,\{1, 2, \ldots, q_i \} \ \hbox{such that} \ X_{\pm 1} \neq X_0.
\end{equation}
 The second ingredient is to observe that, given again the conditioning in the statement of the lemma and using the same argument as in the
 proof of Lemma~\ref{lem:jump}, we have
\begin{equation}
\label{eq:outcome-5}
  \begin{array}{rcl}
    \rw_t (e + 1, i) \ = \ 0 & \hbox{if and only if} & \scalar{\ax_t (e + 1/2)}_i \ = \ \scalar{\ax_t (e + 3/2)}_i \vspace*{4pt} \\
                             & \hbox{if and only if} & \scalar{\ax_{t-} (e - 1/2)}_i \ = \ \scalar{\ax_{t-} (e + 3/2)}_i \vspace*{4pt} \\
                             & \hbox{if and only if} & \scalar{\ax_{t-} (x - 1)}_i \ = \ \scalar{\ax_{t-} (x + 1)}_i \vspace*{4pt} \\
                             & \hbox{if and only if} & \scalar{\ax_0 (a_{t-} (x - 1, i))}_i \ = \ \scalar{\ax_0 (a_{t-} (x + 1, i))}_i \vspace*{4pt} \\
                             & \hbox{if and only if} & X_{-1} \ = \ X_{+1}. \end{array}
\end{equation}
 From~\eqref{eq:outcome-4}--\eqref{eq:outcome-5}, it follows that whether the pair~$(e + 1, i)$ is empty or occupied at time~$t$ is an event independent
 of the realization of the spin system up to time~$t-$ and that
 $$ \begin{array}{l}
      P \,(\rw_t (e + 1, i) = 0 \ | \ (e - 1/2, i) \leadsto_t (e + 1/2, i) \ \hbox{and} \ \rw_{t-} (e + 1, i) = 1) \vspace*{4pt} \\ \hspace*{80pt} = \
      P \,(X_{-1} = X_{+1} \ | \ X_{-1} \neq X_0 \ \hbox{and} \ X_{+1} \neq X_0) \ = \ (q_i - 1)^{-1}. \end{array} $$
 This completes the proof.
\end{proof} \\ \\
 In conclusion, combining Lemmas~\ref{lem:active}--\ref{lem:outcome} and obvious symmetry implying that each of these lemmas extends to arrows
 directed to the left rather than to the right, we obtain the following description of the spin system~\eqref{eq:spin} which basically consists of the
 superposition of non-independent systems of one-dimensional annihilating-coalescing symmetric random walks:
\begin{itemize}
 \item The particles at edge~$e$ jump independently at the same rate~$r (\pile (e))$ one unit to the left or one unit to the right. In particular,
 \begin{itemize}
  \item In case there are~$F$ particles at~$e$, they cannot jump so \\ we call these particles {\bf frozen} particles and the edge a {\bf blockade} at time~$t$.
  \item In case there are less than~$F$ particles at~$e$, they jump at a positive rate so \\ we call these particles {\bf active} particles and the edge a {\bf live edge} at time~$t$.
 \end{itemize}
 \item When a particle jumps onto a pair edge-level~$(e, i)$ which is already occupied, both particles annihilate or coalesce independently
   of the past with respective probabilities
   $$ (q_i - 1)^{-1} \qquad \hbox{and} \qquad 1 - (q_i - 1)^{-1} = (q_i - 2)(q_i - 1)^{-1}. $$
\end{itemize}
 See~Figure~\ref{fig:dynamics} for an illustration of the dynamics.

\begin{figure}[t]
\centering
\scalebox{0.45}{\input{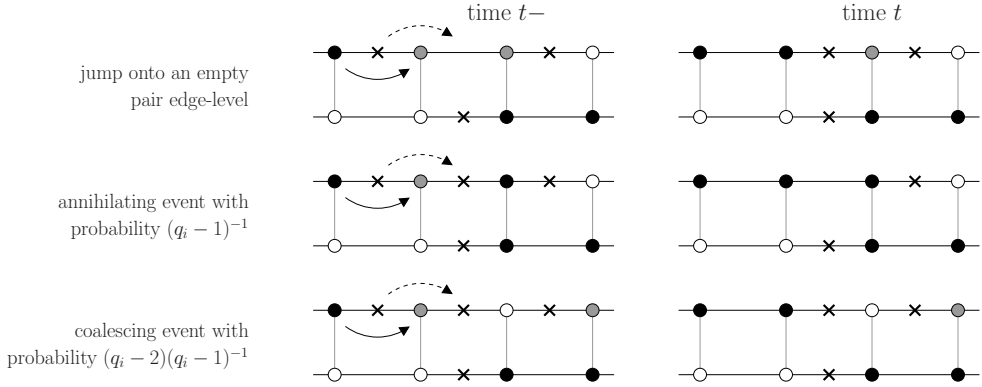}}
\caption{\upshape{Schematic illustration of the dynamics of the spin system~\eqref{eq:spin}.}}
\label{fig:dynamics}
\end{figure}


\section{Sufficient condition for fixation}
\label{sec:fixation}

\indent As previously mentioned, in addition to the coupling with annihilating-coalescing random walks, the first ingredient to prove
 the theorem is a construction due to Bramson and Griffeath~\cite{bramson_griffeath_1989} based on duality-like techniques to obtain an
 implicit condition for fixation in terms of the initial distribution of active and frozen particles in a large interval.
 In this section, we briefly recall their construction and derive a condition for fixation of the Axelrod model
 following~\cite{lanchier_scarlatos_2013}.
\begin{lemma} --
\label{lem:fixation}
 For all $(z, i) \in \Z \times \{1, 2, \ldots F \}$, let
 $$ T (z, i) \ := \ \inf \,\{t : (z, 0) \overset{i}{\leadsto} (0, t) \}. $$
 Then, the system fixates whenever
\begin{equation}
\label{eq:fixation-1}
 \begin{array}{l}
   \lim_{N \to \infty} \,P \,(T (z, i) < \infty \ \hbox{for some} \ z < - N \ \hbox{and some} \ i = 1, 2, \ldots, F) \ = \ 0.
 \end{array}
\end{equation}
\end{lemma}
\begin{proof}
 This follows exactly the proof of~\cite[Lemma 4]{lanchier_scarlatos_2013}.
\end{proof} \\ \\
 To make this condition for fixation more explicit, the idea is to study the connection between the initial configuration
 of the system and the key event
 $$ H_N \ := \ \{T (z, i) < \infty \ \hbox{for some} \ z < - N \ \hbox{and some} \ i = 1, 2, \ldots, F \} $$
 that appears in~\eqref{eq:fixation-1}.
 More precisely, we prove that, on the event~$H_N$, there is an arbitrarily large random interval such that all the blockades initially in this
 interval must have been destroyed by either active particles initially in this interval or active particles that result from the destruction
 of these blockades.
 To make this construction precise, we let
 $$ \tau_N \ = \ \inf \,\{T (z, i) : z \in (- \infty, - N) \ \hbox{and} \ i = 1, 2, \ldots, F \} $$
 be the first time an active $i$-path that originates from the interval $(- \infty, - N)$ hits the origin and observe that the
 event~$H_N$ can be written as
 $$ H_N \ = \ \{T (z, i) < \infty \ \hbox{for some} \ (z, i) \in (- \infty, - N) \times \{1, 2, \ldots, F \} \} \ = \ \{\tau_N < \infty \}. $$
 Note that, though active paths at the same level cannot cross each other, active paths at different levels can.
 To construct an interval where the initial blockades cannot be destroyed by active particles that originate from outside this
 interval, we need to extend our definition:
 we say that there is a {\bf generalized active path} from~$(z, s)$ to~$(x, t)$ whenever there exist
 $$ s_0 \ = \ s \ < \ s_1 \ < \ \cdots \ < \ s_{n + 1} \ = \ t \qquad \hbox{and} \qquad
    x_0 \ = \ z, \,x_1, \,\ldots, \,x_n \ = \ x $$
 such that the following two conditions hold:
\begin{enumerate}
 \item For all $j = 1, 2, \ldots, n$, there exists~$i = i (j)$ such that $(x_{j - 1}, i) \leadsto_{s_j} (x_j, i)$. \vspace*{4pt}
 \item For all $j = 0, 1, \ldots, n$, $s \in (s_j, s_{j + 1})$ and~$i$, there is no active arrow $(x_{j - 1}, i) \leadsto_s (x_j, i)$.
\end{enumerate}
 We write this event~$(z, s) \leadsto (x, t)$ and observe that, due to one-dimensional nearest neighbor interactions, generalized active paths cannot
 cross each other.
 To construct the random interval mentioned above given the event~$H_N$, we define the random variables
\begin{equation}
\label{eq:paths}
  \begin{array}{rcl}
    l_N & := & \min \,\{z \in \Z : (z, 0) \leadsto (0, \tau_N) \} \ < \ - N \vspace*{2pt} \\
    r_N & := & \max \,\{z \in \Z : (z, 0) \leadsto (0, \sigma_N) \ \hbox{for some} \ \sigma_N < \tau_N \} \ \geq \ 0 \end{array}
\end{equation}
 and set~$I_N := (l_N, r_N) \subset (-N, 0)$.
 This construction implies that
\begin{itemize}
\item All the blockades initially in~$I_N$ must break, i.e., each pile of frozen particles is turned into a smaller pile of active particles
      due to an annihilating event, by time~$\tau_N$. \vspace*{4pt}
\item The active particles initially outside~$I_N$ cannot jump inside the space-time region delimited by the two generalized
      active paths defined implicitly in~\eqref{eq:paths} since, due to one-dimensional nearest neighbor interactions, generalized active paths
      cannot cross each other.
\end{itemize}
 In particular, given the event~$H_N$, all the blockades initially in the interval~$I_N$ must have been destroyed by either active particles
 initially in this interval or active particles that result from the destruction of these blockades.
 To keep track of and count these particles, we attribute a weight to each edge based on the number of particles it carries initially.
 To begin with, we give an arbitrary weight, say weight~$-1$, to each particle initially active by setting
\begin{equation}
\label{eq:weight-active}
 \weight (e) \ := \ - i \quad \hbox{whenever} \quad \pile_0 (e) = i \neq F.
\end{equation}
 To define the weight of a blockade, we observe that, before the blockade breaks due to an annihilating event, which occurs almost
 surely on the event~$H_N$, a random number of active particles have disappeared due to successive coalescing events with the blockade.
 Moreover, the destruction of this blockade results in~$F - 1$ frozen particle becoming active so we let
\begin{equation}
\label{eq:breaks}
  T_e \ := \ \inf \,\{t > 0 : \pile_t (e) \neq F \}
\end{equation}
 and define the weight of a blockade initially at~$e$ as
\begin{equation}
\label{eq:weight-blockade}
 \weight (e) \ := \ - (F - 1) + \hbox{number of particles that hit~$e$ until time~$T_e$}.
\end{equation}
 The fact that the occurrence of~$H_N$ implies that all the blockades initially in~$I_N$ must have been destroyed by either active
 particles initially in this interval or active particles that result from the destruction of these blockades can then be written as
\begin{equation}
\label{eq:inclusion}
 \begin{array}{rcl}
  H_N \ \subset \ \{\sum_{e \in I_N} \weight (e) \leq 0 \} \end{array}
\end{equation}
 which will be used in the next sections to prove fixation.


\section{Number of collisions to break a blockade}
\label{sec:collisions}

\indent Starting from this section, we focus on the case~$F = 2$ where the process is referred to as the two-feature Axelrod model.
 Motivated by the results of the previous section, the main objective is to find a stochastic lower bound for the number of collisions of an
 active particle with a blockade before this one breaks, and therefore a lower bound for the weight of a blockade.
 To begin with, we observe that, according to~Lemma~\ref{lem:outcome}, the number of collisions different particles undergo before they
 annihilate are independent geometric random variables whose success parameter only depends on the number of states at the corresponding level.
 In particular, assuming that~$e$ is initially a blockade and letting~$T_e$ be defined as in~\eqref{eq:breaks}, we have for $q_1 > q_2$
\begin{equation}
\label{eq:bounds}
  Y_1 \ \succeq \ \card \,\{t \leq T_e : (e \pm 3/2, i) \leadsto_t (e \pm 1/2, i) \ \hbox{for some} \ i = 1, 2 \} \ \succeq \ Y_2
\end{equation}
 where $Y_i = \geometric ((q_i - 1)^{-1})$ for $i = 1, 2$, are independent, and where~$\succeq$ means stochastically larger than.
 To improve the lower bound, the idea is to show that the density of active particles, and by translation invariance the number of collisions
 per edge per unit of time, at level~1 is always above the density of active particles at level~2.
 From this key result, we will deduce that the number of collisions to break a blockade is stochastically larger than a certain convex
 combination of the geometric random variables~$Y_1$~and~$Y_2$.
 To make the argument rigorous, we let
 $$ \begin{array}{rcl}
    \bar u_i (t) & := & P \,(\rw_t (e, i) = 1) \ = \ \hbox{density of particles at level~$i$ for $i = 1, 2$,} \vspace*{2pt} \\
         u_1 (t) & := & P \,(\rw_t (e, 1) > \rw_t (e, 2)) \ = \ \hbox{density of active particles at level~1,} \vspace*{2pt} \\
         u_2 (t) & := & P \,(\rw_t (e, 1) < \rw_t (e, 2)) \ = \ \hbox{density of active particles at level~2,} \end{array} $$
 and observe that, since the initial distribution and the graphical representation are translation invariant, these functions do not depend on the
 choice of~$e$ and are well-defined.
 The next lemma gives a monotonicity property between the densities and the number of states.
\begin{lemma} --
\label{lem:density}
 Assume that $q_1 > q_2$. Then, $u_1 (t) \geq u_2 (t)$ for all $t \geq 0$.
\end{lemma}
\begin{proof}
 Since there is a frozen particle at~$e$ at one level if and only if there is a frozen particle at~$e$ at the other level, the density of frozen particles
 is the same at both levels at all times:
\begin{equation}
\label{eq:density-1}
  \bar u_1 (t) - u_1 (t) \ = \ \bar u_2 (t) - u_2 (t) \quad \hbox{for all} \quad t \geq 0.
\end{equation}
 In view of the initial distribution of the system, we also have
 $$ \bar u_1 (0) \ = \ P \,(\ax_0 (e - 1/2, 1) \neq \ax_0 (e + 1/2, 1)) \ = \ 1 - q_1^{-1} \ > \ 1 - q_2^{-1} \ = \ \bar u_2 (0) $$
 which, together with~\eqref{eq:density-1}, implies that the inequality to be proved holds at time~0.
 Now, assume by contradiction that this inequality is not true at some time~$t > 0$.
 Since~$\bar u_i$ is nonincreasing because the particles can only coalesce or annihilate, and since for all times~$h > 0$ small
 $$ \begin{array}{rcl}
    \bar u_i (s + h) - \bar u_i (s) & \leq & E \,(\rw_{s + h} (e, i)) - E \,(\rw_s (e, i))
                                    \ \leq \ P \,(\rw_{s + h} (e, i) > \rw_s (e, i)) \vspace*{4pt} \\
                                    & \leq & P \,((e \pm 3/2, i) \to_{s'} (e \pm 1/2, i) \ \hbox{for some} \ s' \in (s, s + h)) \vspace*{4pt} \\
                                    & \leq & 1 - e^{-h} \ = \ h + o (h), \end{array} $$
 the function~$\bar u_i$ is continuous.
 In particular, it follows from our assumption and the intermediate value theorem that there exists a time~$s_0 < t$ such that
\begin{equation}
\label{eq:density-2}
  \bar u_1 (s_0) \ = \ \bar u_2 (s_0) \quad \hbox{and} \quad \bar u_1 (s) \ < \ \bar u_2 (s) \quad \hbox{for all} \quad s \in (s_0, t).
\end{equation}
 In words, the density of particles at time~$s_0$ is the same at both levels, therefore the density of active particles is the same at both
 levels and the density of frozen particles is the same at both levels according to~\eqref{eq:density-1}.
 Since in addition all the active particles jump at the same rate when there are only two levels, the expected number of collisions per unit of time
 at a given edge is also the same at each level.
 But according to Lemma~\ref{lem:outcome}, because~$q_1 > q_2$, the collisions at level~1 are less likely to result in annihilation and more
 likely to result in coalescence, and therefore remove in average less particles, than the ones at level~2.
 In particular, there exists~$\ep > 0$ such that
 $$ \bar u_1 (s) \ \geq \ \bar u_2 (s) \quad \hbox{for all} \quad s \in (s_0, s_0 + \ep) $$
 in contradiction with~\eqref{eq:density-2}.
 In conclusion, the density of particles is always larger at level~1 than at level~2 and the lemma follows by using~\eqref{eq:density-1} once more.
\end{proof} \\ \\
 Combining~Lemmas~\ref{lem:outcome}~and~\ref{lem:density}, we can now prove the main result of this section, which improves the stochastic lower
 bound in~\eqref{eq:bounds} when the number of states are different: $q_1 \neq q_2$.
\begin{lemma} --
\label{lem:collisions}
 Assume that $\pile_0 (e) = 2$ and $q_1 > q_2$. Then,
 $$ \card \,\{t \leq T_e : (e \pm 3/2, i) \leadsto_t (e \pm 1/2, i) \ \hbox{for some} \ i = 1, 2 \} \ \succeq \ (1/2)(Y_1 + Y_2). $$
\end{lemma}
\begin{proof}
 First, we let $Z_i$ be the number of collisions the particle originally at $(e, i)$ undergoes before it annihilates.
 In view of Lemma~\ref{lem:outcome}, each collision of two particles results independently in their annihilation with probability~$(q_i - 1)^{-1}$
 therefore $Z_i = Y_i$ in distribution:
\begin{equation}
\label{eq:collisions-1}
   P \,(Z_i > n) \ = \ P \,(Y_i > n) \ = \ (q_i - 2)^n (q_i - 1)^{-n} \quad \hbox{for} \quad i = 1, 2.
\end{equation}
 In addition, since the density of active particles at level~1 is larger than the density of active particles at level~2 according to Lemma~\ref{lem:density},
 that the active particles all jump at the same rate, and that the distribution of particles is translation invariant, the number of collisions with the
 blockade at~$e$ before it breaks is stochastically larger at level~1 than at level~2, that is,
\begin{equation}
\label{eq:collisions-2}
 \begin{array}{l}
  \card \,\{t \leq T_e : (e \pm 3/2, 1) \leadsto_t (e \pm 1/2, 1) \} \vspace*{4pt} \\ \hspace*{40pt} \succeq \
  \card \,\{t \leq T_e : (e \pm 3/2, 2) \leadsto_t (e \pm 1/2, 2) \}
 \end{array}
\end{equation}
 The lemma directly follows from~\eqref{eq:collisions-1}--\eqref{eq:collisions-2}.
\end{proof}


\section{Fixation when $q_1 + q_2 \geq 7$}
\label{sec:fixation-1}

\indent In this section, we combine~Lemma~\ref{lem:collisions} and~\eqref{eq:inclusion} to prove a weak version of the theorem.
 Another key to obtaining an explicit condition for fixation is the use of the law of large numbers in order to show that the asymptotic probabilities of the
 event in~\eqref{eq:inclusion} can be studied by simply looking at whether the expected value of the contribution of a typical edge is positive or negative.
 This idea will be used again in the next section together with the ergodic theorem as well as additional arguments to prove the full theorem.
 To begin with, we let
 $$ Y_i (e) \ = \ \geometric ((q_i - 1)^{-1}) \quad \hbox{for all} \quad i = 1, 2 \quad \hbox{and} \quad e \in \Z + 1/2 $$
 be independent.
 Recalling~\eqref{eq:weight-active}--\eqref{eq:weight-blockade}, we deduce from~Lemma~\ref{lem:collisions} that
\begin{itemize}
 \item $\weight (e) = - i$ almost surely when $\pile_0 (e) = i \neq 2$ and \vspace*{4pt}
 \item $\weight (e) = (1/2)(Y_1 (e) + Y_2 (e)) - 1$ in distribution when $\pile_0 (e) = 2$.
\end{itemize}
 This, together with~\eqref{eq:inclusion}, implies that
\begin{equation}
\label{eq:fixation-2}
  \begin{array}{l} P \,(H_N) \ \leq \ P \,(\scalar{(1/2)(Y_1 + Y_2) - 1, \ind \{\pile_0 = 2 \}}_{I_N} - \scalar{\ind \{\pile_0 = 1 \}}_{I_N} \leq 0). \end{array}
\end{equation}
 where for all~$u, v : \Z + 1/2 \longrightarrow \R$ and~$B \subset \Z + 1/2$
 $$ \begin{array}{rcl}
       \scalar{u}_B & := & (\card B)^{-1} \,\sum_{e \in B} \,u (e) \vspace*{4pt} \\
    \scalar{u, v}_B & := & (\card B)^{-1} \,\sum_{e \in B} \,u (e) \,v (e). \end{array} $$
 To state our next lemma, we also introduce the probabilities
 $$ \begin{array}{rcl}
      p_0 & := & q_1^{-1} \,q_2^{-1} \vspace*{4pt} \\
      p_1 & := & q_1^{-1} \,(1 - q_2^{-1}) + q_2^{-1} \,(1 - q_1^{-1}) \vspace*{4pt} \\
      p_2 & := & (1 - q_1^{-1})(1 - q_2^{-1}). \end{array} $$
\begin{lemma} --
\label{lem:lln}
 We have the convergence
 $$ \begin{array}{l}
    \lim_{N \to \infty} \,(\scalar{(1/2)(Y_1 + Y_2) - 1, \ind \{\pile_0 = 2 \}}_{I_N} - \scalar{\ind \{\pile_0 = 1 \}}_{I_N}) \vspace*{4pt} \\ \hspace*{100pt} = \
                                   (1/2)(q_1 + q_2 - 4) \,p_2 - p_1 \quad \hbox{almost surely}. \end{array} $$
\end{lemma}
\begin{proof}
 In view of the initial distribution, we have
 $$ \begin{array}{l}
      P \,(\pile_0 (e) = 2 \,| \,\scalar{\ax_0 (x)}_i \ \hbox{for} \ x < e \ \hbox{and} \ i = 1, 2) \vspace*{4pt} \\ \hspace*{25pt} = \
      P \,(\scalar{\ax_0 (e + 1/2)}_i \neq \scalar{\ax_0 (e - 1/2)}_i \ \hbox{for} \ i = 1, 2  \,| \,\scalar{\ax_0 (x)}_i \ \hbox{for} \ x < e \ \hbox{and} \ i = 1, 2) \vspace*{4pt} \\ \hspace*{25pt} = \
     (1 - q_1^{-1})(1 - q_2^{-1}) \ = \ p_2 \ \ \hbox{almost surely} \end{array} $$
 from which it follows that the initial number of blockades in a given finite interval is a binomial random variable with success probability~$p_2$.
 Therefore, by the law of large numbers,
\begin{equation}
\label{eq:lln-1}
  \begin{array}{l} \lim_{N \to \infty} \,\scalar{\ind \{\pile_0 = 2 \}}_{I_N} \ = \ P \,(\pile_0 (e) = 2) \ = \ p_2 \quad \hbox{almost surely}. \end{array}
\end{equation}
 Similarly, the initial number of active particles in a given finite interval is a binomial random variable with success probability given by
 $$ \begin{array}{rcl}
      P \,(\pile_0 (e) = 1) & = &
      P \,(\card \,\{i : \scalar{\ax_0 (e - 1/2)}_i = \scalar{\ax_0 (e + 1/2)} \} = 1) \vspace*{4pt} \\ & = &
      q_1^{-1} \,(1 - q_2^{-1}) + q_2^{-1} \,(1 - q_1^{-1}) \ = \ p_1 \end{array} $$
 therefore the law of large numbers implies that
\begin{equation}
\label{eq:lln-2}
  \begin{array}{l} \lim_{N \to \infty} \,\scalar{\ind \{\pile_0 = 1 \}}_{I_N} \ = \ P \,(\pile_0 (e) = 1) \ = \ p_1 \quad \hbox{almost surely}. \end{array}
\end{equation}
 Since the geometric random variables~$Y_i (e)$ are independent, the law of large numbers again applies, from which it follows that we have the almost sure convergence
\begin{equation}
\label{eq:lln-3}
  \begin{array}{l} \lim_{N \to \infty} \,\scalar{\,Y_i \,}_{I_N} \ = \ E \,(Y_i (e)) \ = \ q_i - 1 \quad \hbox{almost surely}. \end{array}
\end{equation}
 Combining~\eqref{eq:lln-1}~and~\eqref{eq:lln-3} and using that the geometric random variables attached to a given edge are independent
 of the event that this edge is initially a blockade, we deduce
\begin{equation}
\label{eq:lln-4}
  \begin{array}{l}
  \lim_{N \to \infty} \,\scalar{Y_1 + Y_2, \ind \{\pile_0 = 2 \}}_{I_N} \ = \
   E \,(Y_1 (e) + Y_2 (e)) \,P \,(\pile_0 (e) = 2) \vspace*{4pt} \\ \hspace*{80pt} = \
       (q_1 - 1 + q_2 - 1) \,p_2 \ = \ (q_1 + q_2 - 2) \,p_2 \quad \hbox{almost surely}. \end{array}
\end{equation}
 Finally, we combine~\eqref{eq:lln-1}--\eqref{eq:lln-2}~and~\eqref{eq:lln-4} to conclude
 $$ \begin{array}{l}
    \lim_{N \to \infty} \,(\scalar{(1/2)(Y_1 + Y_2) - 1, \ind \{\pile_0 = 2 \}}_{I_N} - \scalar{\ind \{\pile_0 = 1 \}}_{I_N}) \vspace*{4pt} \\ \hspace*{20pt} = \
     (1/2) \,\lim_{N \to \infty} \,\scalar{Y_1 + Y_2, \ind \{\pile_0 = 2 \}}_{I_N} - \lim_{N \to \infty} \,\scalar{\ind \{\pile_0 = 2 \}}_{I_N} - \scalar{\ind \{\pile_0 = 1 \}}_{I_N} \vspace*{4pt} \\ \hspace*{20pt} = \
     (1/2)(q_1 + q_2 - 2) \,p_2 - p_2 - p_1 \ = \ (1/2)(q_1 + q_2 - 4) \,p_2 - p_1 \end{array} $$
 almost surely.
 This completes the proof.
\end{proof}
\begin{lemma} --
\label{lem:weak-fixation}
 The system fixates whenever~$q_1 + q_2 \geq 7$.
\end{lemma}
\begin{proof}
 Assume that~$q_1 + q_2 \geq 7$. Then,
 $$ \begin{array}{rrl}
    h_1 (q_1, q_2) & := & (1/2)(q_1 + q_2 - 4) \,p_2 - p_1 \vspace*{4pt} \\
                   &  = & (1/2)(q_1 + q_2 - 4)(1 - q_1^{-1})(1 - q_2^{-1}) - q_1^{-1} \,(1 - q_2^{-1}) - q_2^{-1} \,(1 - q_1^{-1}) \vspace*{4pt} \\
                   &  = & (1/2)(q_1 + q_2 - 2)(1 - q_1^{-1})(1 - q_2^{-1}) - (1 - q_1^{-1} \,q_2^{-1}) \vspace*{4pt} \\
                   & \geq & \min \,(h_1 (2, 5), h_1 (3, 4)) \ = \ h_1 (2, 5) \ = \ 1/10. \end{array} $$
 In particular, applying~Lemma~\ref{lem:lln}, we get
 $$ \begin{array}{l}
    \lim_{N \to \infty} \,(\scalar{(1/2)(Y_1 + Y_2) - 1, \ind \{\pile_0 = 2 \}}_{I_N} - \scalar{\ind \{\pile_0 = 1 \}}_{I_N}) \ = \ h_1 (q_1, q_2) \ \geq \ 1/10 \end{array} $$
 almost surely, which implies that
 $$ \begin{array}{l}
    \lim_{N \to \infty} \,P \,(\scalar{(1/2)(Y_1 + Y_2) - 1, \ind \{\pile_0 = 2 \}}_{I_N} - \scalar{\ind \{\pile_0 = 1 \}}_{I_N} \leq 0) \ = \ 0. \end{array} $$
 This, together with Lemma~\ref{lem:fixation} and~\eqref{eq:fixation-2}, implies the lemma.
\end{proof}

\section{Fixation when $q_1 + q_2 \geq 6$}
\label{sec:fixation-2}

\begin{figure}[t]
\centering
\scalebox{0.45}{\input{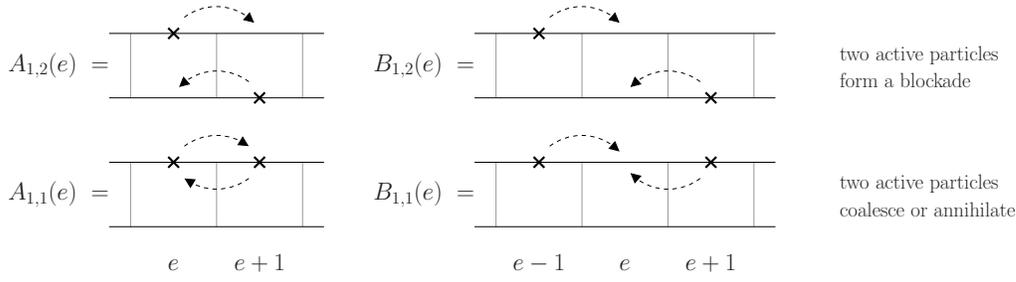}}
\caption{\upshape{Schematic illustration of the events~$\near_{i, j} (e)$ and~$\next_{i, j} (e)$ for $i \neq j$ and $i = j$.}}
\label{fig:events}
\end{figure}

\indent This last section is devoted to proving the theorem.
 To explain the idea behind the proof, note that our definition of weight is based on the worst case scenario where active
 particles \emph{do their best} to destroy the blockades and turn as many frozen particles as possible into active particles.
 This assumption ignores in particular annihilating events involving two active particles, coalescing events involving two active
 particles, and events where two active particles at different levels move towards each other to form a blockade.
 To prove the full theorem, we now take these events into account.
 More precisely, we define the following collections of events:
 $$ \begin{array}{rcl}
     \near_{i, j} (e) & := & \hbox{there is initially two active particles, one at~$(e, i)$ and one at~$(e + 1, j)$,} \\ &&
                              \hbox{and the first update at the corresponding two edges is a jump of one} \\ &&
                              \hbox{of the two active particles toward the other active particle} \vspace*{4pt} \\
     \next_{i, j} (e)  & := & \hbox{there is initially one active particle at~$(e - 1, i)$ and one active} \\ &&
                              \hbox{particle at~$(e + 1, j)$ and the first two updates at the corresponding} \\ &&
                              \hbox{edges are jumps of these two active particles to edge~$e$.} \end{array} $$
 The next lemma gives an improvement of~\eqref{eq:fixation-2} taking into account these events.
\begin{lemma} --
\label{lem:improve}
 We have the inequality
\begin{equation}
\label{eq:improve-1}
  \begin{array}{l}
    P \,(H_N) \ \leq \ P \,(\scalar{(1/2)(Y_1 + Y_2) - 1, \ind \{\pile_0 = 2 \}}_{I_N} - \scalar{\ind \{\pile_0 = 1 \}}_{I_N} \vspace*{4pt} \\ \hspace*{60pt} + \
                            \sum_{i \neq j} \,\scalar{(1/2)(Y_1 + Y_2) + 1, \ind \{\near_{i, j} \} + \ind \{\next_{i, j} \}}_{I_N} \vspace*{4pt} \\ \hspace*{60pt} + \
                            \sum_{i = 1, 2} \,\scalar{\ind \{Y_i > 1 \} + 2 \times \ind \{Y_i = 1 \}, \ind \{\near_{i, i} \} + \ind \{\next_{i, i} \}}_{I_N} \leq 0). \end{array}
\end{equation}
\end{lemma}
\begin{proof}
 Note that the random variable
 $$ \scalar{(1/2)(Y_1 + Y_2) - 1, \ind \{\pile_0 = 2 \}}_{I_N} - \scalar{\ind \{\pile_0 = 1 \}}_{I_N} $$
 is the one on the right-hand side of~\eqref{eq:fixation-2}, which is obtained based on the worst case scenario where all the active particles
 first hit a blockade rather than coalescing or annihilating with another active particle or forming a blockade with another active particle.
 The two sums in~\eqref{eq:improve-1} are correction terms taking these events into account.
 To quantify these corrections, note that
 $$ \begin{array}{rl}
    \near_{i, j} (e) \ \hbox{and} \ \next_{i, j} (e) & \hbox{induce a blockade formation when~$i \neq j$} \vspace*{2pt} \\
                                                      & \hbox{induce the collision of two active particles when~$i = j$}. \end{array} $$
 In particular, the random variables
\begin{equation}
\label{eq:improve-2}
  \begin{array}{l} \sum_{i \neq j} \,\scalar{\ind \{\near_{i, j} \} + \ind \{\next_{i, j} \}}_{I_N} \quad \hbox{and} \quad
                   \sum_{i = j}    \,\scalar{\ind \{\near_{i, j} \} + \ind \{\next_{i, j} \}}_{I_N} \end{array}
\end{equation}
 are stochastic lower bounds for the fraction of blockade formations and the fraction of collisions of active particles that occur in the
 space-time region delimited by the two generalized active paths defined implicitly in~\eqref{eq:paths}.
 Since, when two active particles form a blockade, their total weight can be replaced by the weight of a blockade, the correction term
 due to blockade formations is
\begin{equation}
\label{eq:improve-3}
 \begin{array}{l}
   \weight (e \,| \,\pile_0 (e) = 2) - 2 \times \weight (e \,| \,\pile_0 (e) = 1) \vspace*{4pt} \\ \hspace*{40pt} = \
   (1/2)(Y_1 (e) + Y_2 (e)) - 1 + 2 \ = \ (1/2)(Y_1 (e) + Y_2 (e)) + 1 \end{array}
\end{equation}
 in distribution.
 In addition, when two active particles collide, their total weight can be decreased by the weight of either one or two active particles
 depending on whether the collision results in a coalescing or an annihilating event, respectively.
 Recalling also from Lemma~\ref{lem:outcome} that the geometric random variable~$Y_i (e)$ counts the number of collisions at level~$i$ until the
 first annihilating event occurs, we deduce that the correction term due to collisions of active particles at level~$i$ is
\begin{equation}
\label{eq:improve-4}
 \begin{array}{l}
   - \ \ind \{\hbox{coalescence} \} \times \weight (e \,| \,\pile_0 (e) = 1) \vspace*{4pt} \\ \hspace*{50pt}
   - \ 2 \times \ind \{\hbox{annihilation} \} \times \weight (e \,| \,\pile_0 (e) = 1) \vspace*{4pt} \\ \hspace*{25pt} = \
     \ind \{Y_i (e) = 1 \} + 2 \times \ind \{Y_i (e) > 1 \} \end{array}
\end{equation}
 in distribution.
 The lemma directly follows from~\eqref{eq:fixation-2} and \eqref{eq:improve-2}--\eqref{eq:improve-4}.
\end{proof} \\ \\
 In the next two lemmas, we compute the limits as~$N \to \infty$ of the random variables on the right-hand side of inequality~\eqref{eq:improve-1}.
 These lemmas can be seen as the analog of~Lemma~\ref{lem:lln}.
 To state these two lemmas, we introduce the probabilities
 $$ p_{11} \ := \ q_2^{-1} \,(1 - q_1^{-1}) \quad \hbox{and} \quad p_{12} \ := \ q_1^{-1} \,(1 - q_2^{-1}). $$
\begin{lemma} --
\label{lem:ergodic-frozen}
 We have the limit
 $$ \begin{array}{l}
    \lim_{N \to \infty} \,\sum_{i \neq j} \,\scalar{(1/2)(Y_1 + Y_2) + 1, \ind \{\near_{i, j} \} + \ind \{\next_{i, j} \}}_{I_N} \vspace*{4pt} \\ \hspace*{80pt} = \
         (1/4)(q_1 + q_2)(1 + (1/8) \,p_0) \,p_{11} \,p_{12} \quad \hbox{almost surely}. \end{array} $$
\end{lemma}
\begin{proof}
 For all~$i, j = 1, 2$ and~$e \in \Z + 1/2$, we define the events
\begin{equation}
\label{eq:ergodic-frozen-1}
  \near_{i, j}' (e) \ := \ \{\pile_0 (e) = \pile_0 (e + 1) = \rw_0 (e, i) = \rw_0 (e + 1, j) = 1 \}
\end{equation}
 which are measurable with respect to the initial configuration of the system, as well as the events measurable with respect to the graphical representation
\begin{equation}
\label{eq:ergodic-frozen-2}
  \begin{array}{rcl}
    \near_{i, j}'' (e) & := & \{\hbox{there is an arrow} \ (e - 1/2, i) \to (e + 1/2, i) \vspace*{2pt} \\ && \hspace*{40pt}
                                 \hbox{or an arrow} \ (e + 3/2, j) \to (e + 1/2, j) \ \hbox{at time} \ T_{i, j} (e) \} \end{array}
\end{equation}
 where time~$T_{i, j} (e)$ is the first time one of the following eight arrows
 $$ \begin{array}{rclcrcl}
      (e - 3/2, 1) & \to & (e - 1/2, 1) & \quad & (e - 3/2, 2) & \to & (e - 1/2, 2) \vspace*{2pt} \\
      (e + 5/2, 1) & \to & (e + 3/2, 1) & \quad & (e + 5/2, 2) & \to & (e + 3/2, 2) \vspace*{2pt} \\
      (e - 1/2, i) & \to & (e + 1/2, i) & \quad & (e + 1/2, i) & \to & (e - 1/2, i) \vspace*{2pt} \\
      (e + 3/2, j) & \to & (e + 1/2, j) & \quad & (e + 1/2, j) & \to & (e + 3/2, j) \end{array} $$
 occurs in the graphical representation.
 Note that
\begin{equation}
\label{eq:ergodic-frozen-3}
  \near_{i, j}' (e) \,\cap \,\near_{i, j}'' (e) \ = \ \near_{i, j} (e) \quad \hbox{for all} \quad i, j = 1, 2 \ \hbox{and} \ e \in \Z + 1/2.
\end{equation}
 Note also that the events in~\eqref{eq:ergodic-frozen-1}--\eqref{eq:ergodic-frozen-2} attached to adjacent edges are not independent so the
 law of large numbers no longer applies.
 However, since the initial distribution is the uniform product measure and since the Poisson processes in the graphical representation are
 independent, we can apply the ergodic theorem and the same argument as in the proof of Lemma~\ref{lem:lln} to get
\begin{equation}
\label{eq:ergodic-frozen-4}
  \begin{array}{l}
    \lim_{N \to \infty} \,\scalar{\ind \{\near_{i, j}' \}}_{I_N}  \ = \ P \,(\near_{i, j}' (e))  \ = \ p_{11} \,p_{12} \vspace*{4pt} \\
    \lim_{N \to \infty} \,\scalar{\ind \{\near_{i, j}'' \}}_{I_N} \ = \ P \,(\near_{i, j}'' (e)) \ = \ 2/8 \ = \ 1/4
  \end{array}
\end{equation}
 almost surely for~$i \neq j$.
 In addition, in view of Lemma~\ref{lem:outcome} and since the graphical representation is independent of the initial distribution,
 for each edge~$e$, the random variables
 $$ Y_1 (e) \quad \hbox{and} \quad Y_2 (e) \quad \hbox{and} \quad \near_{1, 2}' (e) \quad \hbox{and} \quad \near_{1, 2}'' (e) \quad \hbox{are independent}. $$
 This, together with~\eqref{eq:lln-3}~and~\eqref{eq:ergodic-frozen-3}--\eqref{eq:ergodic-frozen-4}, implies that
\begin{equation}
\label{eq:ergodic-frozen-5}
  \begin{array}{l}
    \lim_{N \to \infty} \,\sum_{i \neq j} \,\scalar{(1/2)(Y_1 + Y_2) + 1, \ind \{\near_{i, j} \} \}}_{I_N} \vspace*{4pt} \\ \hspace*{25pt} = \
     2 \times \lim_{N \to \infty} \,\scalar{(1/2)(Y_1 + Y_2) + 1, \ind \{\near_{1, 2}' \} \,\ind \{\near_{1, 2}'' \}}_{I_N} \vspace*{4pt} \\ \hspace*{25pt} = \
     E \,(Y_1 (e) + Y_2 (e) + 2) \,P \,(\near_{1, 2}' (e)) \,P \,(\near_{1, 2}'' (e)) \vspace*{4pt} \\ \hspace*{25pt} = \
     (1/4)(q_1 - 1 + q_2 - 1 + 2) \,p_{11} \,p_{22} \ = \ (1/4)(q_1 + q_2) \,p_{11} \,p_{12} \end{array}
\end{equation}
 almost surely.
 Similarly, we prove that
\begin{equation}
\label{eq:ergodic-frozen-6}
  \begin{array}{l}
    \lim_{N \to \infty} \,\sum_{i \neq j} \,\scalar{(1/2)(Y_1 + Y_2) + 1, \ind \{\next_{i, j} \} \}}_{I_N} \vspace*{4pt} \\ \hspace*{25pt} = \
     (1/8) \,p_0 \,\lim_{N \to \infty} \,\sum_{i \neq j} \,\scalar{(1/2)(Y_1 + Y_2) + 1, \ind \{\next_{i, j} \} \}}_{I_N} \vspace*{4pt} \\ \hspace*{25pt} = \
     (1/32)(q_1 + q_2) \,p_0 \,p_{11} \,p_{12} \end{array}
\end{equation}
 The additional factor~$p_0$ is the probability that edge~$e$ is initially empty while the~$1/8$ comes from the requirement in the graphical
 representation: after one of the two active particles jumps, we still need the next update to be the other active particle
 moving to edge~$e$, which occurs if and only if one among eight arrows occurs first.
 The lemma follows by adding~\eqref{eq:ergodic-frozen-5}--\eqref{eq:ergodic-frozen-6}.
\end{proof}
\begin{lemma} --
\label{lem:ergodic-active}
 For~$i = 1, 2$, we have the limit
 $$ \begin{array}{l}
    \lim_{N \to \infty} \,\scalar{\ind \{Y_i > 1 \} + 2 \times \ind \{Y_i = 1 \}, \ind \{\near_{i, i} \} + \ind \{\next_{i, i} \}}_{I_N} \vspace*{4pt} \\ \hspace*{80pt} = \
         (1/4) \,q_i \,(q_i - 1)^{-1} \,(1 + (1/8) \,p_0)(p_{1i})^2 \quad \hbox{almost surely}. \end{array} $$
\end{lemma}
\begin{proof}
 Using the events introduced in~\eqref{eq:ergodic-frozen-1}--\eqref{eq:ergodic-frozen-2} and applying again the ergodic theorem as in the
 previous lemma, we first prove that
\begin{equation}
\label{eq:ergodic-active-1}
  \begin{array}{l}
    \lim_{N \to \infty} \,\scalar{\ind \{\near_{i, j}' \}}_{I_N}  \ = \ P \,(\near_{i, j}' (e))  \ = \ (p_{1i})^2 \vspace*{4pt} \\
    \lim_{N \to \infty} \,\scalar{\ind \{\near_{i, j}'' \}}_{I_N} \ = \ P \,(\near_{i, j}'' (e)) \ = \ 2/8 \ = \ 1/4 \end{array}
\end{equation}
 almost surely for~$i = 1, 2$.
 Using again \eqref{eq:ergodic-frozen-3} and \eqref{eq:ergodic-active-1} instead of~\eqref{eq:ergodic-frozen-4} together with the independence
 of the random variables, we deduce that
\begin{equation}
\label{eq:ergodic-active-2}
  \begin{array}{l}
    \lim_{N \to \infty} \,\scalar{\ind \{Y_i > 1 \} + 2 \times \ind \{Y_i = 1 \}, \ind \{\near_{i, i} \} \}}_{I_N} \vspace*{4pt} \\ \hspace*{25pt} = \
    \lim_{N \to \infty} \,\scalar{\ind \{Y_i > 1 \} + 2 \times \ind \{Y_i = 1 \}, \ind \{\near_{i, i}' \} \,\ind \{\near_{i, i}'' \}}_{I_N} \vspace*{4pt} \\ \hspace*{25pt} = \
     (P \,(Y_i (e) > 1) + 2 \times P \,(Y_i (e) = 1)) \ P \,(\near_{i, i}' (e)) \,P \,(\near_{i, i}'' (e)) \vspace*{4pt} \\ \hspace*{25pt} = \
     (1/4)((q_i - 2)(q_i - 1)^{-1} + 2 \,(q_i - 1)^{-1})(p_{1i})^2 \ = \ (1/4) \,q_i \,(q_i - 1)^{-1} \,(p_{1i})^2 \end{array}
\end{equation}
 almost surely, and for the same reasons as in~Lemma~\ref{lem:ergodic-frozen},
\begin{equation}
\label{eq:ergodic-active-3}
  \begin{array}{l}
    \lim_{N \to \infty} \,\scalar{\ind \{Y_i > 1 \} + 2 \times \ind \{Y_i = 1 \}, \ind \{\next_{i, i} \} \}}_{I_N} \vspace*{4pt} \\ \hspace*{25pt} = \
     (1/8) \,p_0 \,\lim_{N \to \infty} \,\scalar{\ind \{Y_i > 1 \} + 2 \times \ind \{Y_i = 1 \}, \ind \{\near_{i, i} \} \}}_{I_N} \vspace*{4pt} \\ \hspace*{25pt} = \
     (1/32) \,q_i \,(q_i - 1)^{-1} \,p_0 \,(p_{1i})^2 \end{array}
\end{equation}
 almost surely.
 The lemma follows from~\eqref{eq:ergodic-active-2}--\eqref{eq:ergodic-active-3}.
\end{proof} \\ \\
 In view of Lemma~\ref{lem:weak-fixation}, the last step to complete the proof of the theorem is to show fixation under the
 assumption~$q_1 + q_2 = 6$, which is done
 in the next lemma.
\begin{lemma} --
\label{lem:strong-fixation}
 The system fixates whenever~$q_1 + q_2 = 6$.
\end{lemma}
\begin{proof}
 Using the same argument as in the proof of~Lemma~\ref{lem:weak-fixation} together with Lemmas~\ref{lem:improve}--\ref{lem:ergodic-active},
 we deduce that the system with parameters~$q_1$~and~$q_2$ fixates whenever
 $$ \begin{array}{rrl}
    h_2 (q_1, q_2) & := & (1/2)(q_1 + q_2 - 4) \,p_2 - p_1 + (1/4)(1 + (1/8) \,p_0) \vspace*{4pt} \\
                   &    & \hspace*{10pt} \times \ ((q_1 + q_2) \,p_{11} \,p_{12} + q_1 \,(q_1 - 1)^{-1} \,(p_{11})^2 + q_2 \,(q_2 - 1)^{-1} \,(p_{12})^2) \ > \ 0. \end{array} $$
 But a straightforward calculation shows that, when~$q_ 1 + q_2 = 6$,
 $$ \begin{array}{rrl}
    h_2 (q_1, q_2) \ \geq \ h_2 (2, 4) & = & - 16/128 + (1 + 1/64) \times (9/128 + 6/128 + 1/128) \vspace*{4pt} \\
                                       & = & (1/64)(9/128 + 6/128 + 1/128) \ = \ (1/64)(1/8) \ > \ 0. \end{array} $$
 This shows the lemma and completes the proof of the theorem.
\end{proof}


\end{document}